\documentclass[final]{siamart190516}

\usepackage[utf8]{inputenc}
\usepackage[T1]{fontenc}
\usepackage[english]{babel}
\usepackage{amsfonts}
\usepackage{amsopn}
\usepackage[shortlabels]{enumitem}
\usepackage{graphicx}
\usepackage{booktabs}
\usepackage{tabularx}
\usepackage{geometry}
\usepackage{makeidx}
\usepackage{fancyhdr}
\usepackage[autostyle,italian=guillemets]{csquotes}
\usepackage{cite}
\usepackage{courier}
\usepackage{alltt}
\usepackage{colonequals}
\usepackage{ntheorem}
\usepackage{theoremref}
\usepackage{xargs} 
\usepackage[]{xcolor}
\usepackage{multirow}
\usepackage{multicol}
\usepackage[colorinlistoftodos, prependcaption, textsize = small]{todonotes}
\usepackage{tikz}
\usepackage{pgfplots}
\pgfplotsset{compat=1.18}
\usepackage{siunitx}
\usepackage{stackengine}

\usepackage{hyperref}

\stackMath
\newcommand\tsup[2][2]{%
 \def\useanchorwidth{T}%
  \ifnum#1>1%
    \stackon[-.5pt]{\tsup[\numexpr#1-1\relax]{#2}}{\scriptscriptstyle\sim}%
  \else%
    \stackon[.5pt]{#2}{\scriptscriptstyle\sim}%
  \fi%
}

\renewcommand{\ll}{\mathbf l}
\newcommand{\nn}{\mathbf n}
\newcommand{\rr}{\mathbf r}
\renewcommand{\ss}{\mathbf s}
\newcommand{\uu}{\mathbf u}
\newcommand{\ww}{\mathbf w}

\renewcommand{\AA}{\mathbf A}
\newcommand{\BB}{\mathbf B}
\newcommand{\DD}{\mathbf D}
\newcommand{\HH}{\mathbf H}
\newcommand{\MM}{\mathbf M}
\newcommand{\RR}{\mathbf R}
\renewcommand{\SS}{\mathbf S}
\newcommand{\TT}{\mathbf T}
\newcommand{\WW}{\mathbf W}
\newcommand{\MMe}{\MM^+}
\newcommand{\MMo}{\MM^-}
\newcommand{\SSe}{\SS^+}
\newcommand{\SSo}{\SS^-}

\newcommand{\R}{\mathcal R}
\newcommand{\Rh}{\R_h}
\renewcommand{\S}{\mathcal S}

\newcommand{\RRR}{\mathbb R}

\newcommand{\p}{\prime}
\renewcommand{\d}{\partial}
\newcommand{\duality}[2]{{\langle #1, #2 \rangle}}
\newcommand{\inner}[2]{{\left( #1, #2 \right)}}
\newcommand{\norm}[1]{{\left\Vert #1 \right\Vert}}
\newcommand{\normM}[1]{{\left\Vert #1 \right\Vert}_{\sigma_t}}
\newcommand{\norminf}[1]{{\left\Vert #1 \right\Vert}_{\infty}}
\newcommand{\normc}{\norm{\cdot}}

\newcommand{\TRh}{\mathcal{T}_h^R}
\newcommand{\TSh}{\mathcal{T}_h^S}
\newcommand{\Xh}{X_h}
\newcommand{\Sh}{S_h}
\newcommand{\Wh}{W_h}
\newcommand{\WhN}{W_{h,N}}

\newcommand{\nse}{n_S^+}
\newcommand{\nso}{n_S^-}
\newcommand{\nre}{n_R^+}
\newcommand{\nro}{n_R^-}

\def\bMSpheree{\boldsymbol{\mathsf{M}}\!^{+}}
\def\bMSphereo{\boldsymbol{\mathsf{M}}\!^{-}}
\def\bASphere{\boldsymbol{\mathsf{A}}}

\newcommand{\bThetae}{\mathbf{\Theta}^+}
\newcommand{\bThetao}{\mathbf{\Theta}^-}
\newcommand{\bMsco}{\MM_{\sigma_s}^-}
\newcommand{\bMsce}{\MM_{\sigma_s}^+}
\newcommand{\bMtro}{\MM_{\sigma_t}^-}
\newcommand{\bMtre}{\MM_{\sigma_t}^+}

\newsiamremark{remark}{Remark}
\newsiamthm{problem}{Problem}
\newsiamremark{assumption}{Assumption}

\title{On accelerated iterative schemes for anisotropic radiative transfer using residual minimization}

\author{Riccardo Bardin \thanks{Department of Applied Mathematics, University of Twente, P.O. Box 217, 7500 AE Enschede, The Netherlands. \email{r.bardin@utwente.nl}.}
\and Matthias Schlottbom \thanks{Department of Applied Mathematics, University of Twente, P.O. Box 217, 7500 AE Enschede, The Netherlands. \email{m.schlottbom@utwente.nl}}}

\headers{Accelerated iterative schemes for RTE with anisotropic scattering}{R. Bardin, M. Schlottbom}

\begin{document}
\maketitle

\begin{abstract}
    We consider the iterative solution of anisotropic radiative transfer problems using residual minimization over suitable subspaces. We show convergence of the resulting iteration using Hilbert space norms, which allows us to obtain algorithms that are robust with respect to finite-dimensional realizations via Galerkin projections. We investigate in particular the behavior of the iterative scheme for discontinuous Galerkin discretizations in the angular variable in combination with subspaces that are derived from related diffusion problems. The performance of the resulting schemes is investigated in numerical examples for highly anisotropic scattering problems with heterogeneous parameters.  
\end{abstract}

\begin{keywords}
	anisotropic radiative transfer, iterative solution, nonlinear preconditioning, convergence
\end{keywords}

\begin{AMS}
	65F08, 65F10, 65N22, 65N30, 65N45 
\end{AMS}

\section{Introduction}
\label{sec:Introduction}
The radiative transfer equation serves as a fundamental tool in predicting the interaction of electromagnetic radiation with matter, modeling  scattering, absorption and emission. As such, it has a key role in many scientific and societal applications, including medical imaging and tumor treatment \cite{AS09, HIS06}, energy efficient generation of white light \cite{RPBSN18}, climate sciences \cite{E98, STS17}, geosciences \cite{MWTLS17}, and astrophysics \cite{R11}. The stationary monochromatic radiative transfer equation is an integro-differential equation of the form
\begin{equation}
	\label{eq:RTE}
	\ss \cdot \nabla u(\rr,\ss) + \sigma_t(\rr) u(\rr,\ss) = \sigma_s(\rr)\int_{S^{d-1}} \theta(\ss\cdot\ss^{\p}) u(\rr,\ss^{\p})\,d\ss^{\p} + q(\rr,\ss) \quad \text{for } (\rr,\ss)\in  D:=R\times S^{d-1},
\end{equation}
where the specific intensity $ u = u(\rr,\ss) $ depends on the spatial coordinate $ \rr \in R\subset \RRR^d $ ($d=3$ for most practical applications) and on the direction $ \ss \in S^{d-1} $, with $S^{d-1}$ denoting the unit sphere in $\RRR^d$. The gradient appearing in \cref{eq:RTE} is taken with respect to $\rr$ only.
The physical properties of the medium covered by $R$ enter \eqref{eq:RTE} through the total attenuation (or transport) coefficient $ \sigma_t(\rr) := \sigma_a(\rr) + \sigma_s(\rr) $, which accounts for the absorption and scattering rates, respectively, and through the scattering kernel $ \theta(\ss\cdot\ss^{\p})$, which describes the probability of scattering from direction $\ss^{\p}$ into direction $\ss$. Internal sources of radiation are modeled by the function $q(\rr,\ss)$.
We complement \cref{eq:RTE} by non-homogeneous inflow boundary conditions
\begin{align}
    \label{eq:RTE_boundary_condition}
    u(\rr,\ss) = q_\d(\rr, \ss) \quad \text{for } (\rr,\ss) \in \d D_- \colonequals \{ (\rr,\ss)\in \d R \times S^{d-1}: \nn(\rr) \cdot \ss < 0 \},
\end{align}
with incoming intensity specified by $q_\d$.
Here $ \nn(\rr) $ denotes the outward normal unit vector field for a point $ \rr \in \d R $.
We refer to \cite{CZ67} for further details on the derivation of the radiative transfer equation. If $\theta(\ss\cdot\ss^\p) = 1/|S^{d-1}|$, scattering is called isotropic; otherwise, anisotropic.

\subsection{Approach and contribution}
\label{sec:Approach}
A common approach for showing well-posedness of \cref{eq:RTE} is to prove convergence of the following iterative scheme: Given $z_0$, compute the solution $z_{k+1}$ to
\begin{equation}
    \label{eq:SI}
    \ss \cdot \nabla z_{k+1} + \sigma_t z_{k+1} = \sigma_s \int_{S^{d-1}} \theta(\ss\cdot\ss^\p)z_k\,d\ss^\p + q \quad \text{ in }D,
\end{equation}
with $z_{k+1}=q_\d$ on $\d D_-$ for $k\geq 0$ \cite{AL02}. 
Under the condition that $\rho\colonequals\sup_{\rr}\sigma_s(\rr)/\sigma_t(\rr)<1$ one obtains linear convergence of $z_k$ towards $u$ with rate $\rho$ \cite{CZ63}; for more general conditions, see also \cite{ES13}.
Since in many applications mentioned above $\rho\approx 1$, the convergence of $z_k$ to $u$ is prohibitively slow. From a numerical point of view, \cref{eq:SI} serves as a starting point for constructing iterative solvers for discretizations of \cref{eq:RTE}.

In this paper we propose to accelerate the convergence of \cref{eq:SI} through residual minimization over suitable subspaces. 
Specifically, in analogy to \cref{eq:SI}, given $u_k$, we compute in a first step the solution $u_{k+1/2}$ to
\begin{equation}
    \label{eq:SI2}
    \ss \cdot \nabla u_{k+1/2} + \sigma_t u_{k+1/2} = \sigma_s \int_{S^{d-1}} \theta(\ss\cdot\ss^\p)u_k\,d\ss^\p + q, \quad \text{ in }D,
\end{equation}
with $u_{k+1/2}=q_\d$ on $\d D_-$. To proceed, let us introduce the residual 
\begin{equation}
    \label{eq:residual}
    \tilde r_k := \tilde \R(u_k) := q - \left(\ss \cdot \nabla u_k + \sigma_tu_k -  \sigma_s\int_{S^{n-1}} \theta(\ss\cdot\ss^\p)u_k\,d\ss^\p \right),
\end{equation}
and the preconditioned residual $r_k:=\R(u_k)$ that is defined as the solution to
\begin{equation}
    \label{eq:preconditioned_residual}
    \ss\cdot\nabla r_k + \sigma_t r_k = \tilde r_k \quad\text{ in } D,\quad r_k=0 \text{ on }\d D_-. 
\end{equation}
We note that $\R(u_k)=u_{k+1/2}-u_k$.
Using the weighted $L^2$-norm $\normM{v} \colonequals \| \sqrt{\sigma_t} v \|_{L^2(D)}$, we will show in \Cref{lem:contractivity_residuals_operator} below the following monotonicity result for the preconditioned residual.
\begin{lemma}
    \label{lem:contractivity_residuals}
    For $k\geq 0$ let $u_k,\ u_{k+1/2}$ be related by \cref{eq:SI2} and denote $r_k=\R(u_k)$ and $r_{k+1/2}=\R(u_{k+1/2})$ the respective preconditioned residuals defined in \cref{eq:preconditioned_residual}. Then it holds that
    \begin{align*}
        \normM{r_{k+1/2}} \leq \rho \normM{r_k}.
    \end{align*}
\end{lemma}
In view of the monotonicity of the residuals, we look for a correction to the intermediate iterate $u_{k+1/2}$ by residual minimization, i.e., 
\begin{equation}
    \label{eq:minimization_problem}
    u_{k+1/2}^c \colonequals \underset{v \in W_{N}}{\mathrm{argmin}} \normM{\R(u_{k+1/2}+v)},
\end{equation}
where $W_{N}$ is a suitable finite-dimensional linear space of dimension $N$, and set
\begin{align}
    \label{eq:new_iterate}
    u_{k+1} \colonequals u_{k+1/2} + u_{k+1/2}^c.
\end{align}
Using the minimization property and the equivalence between the residual and the error, we will show our main convergence statement.
\begin{theorem}
  \label{thm:main_result}
   For any initial guess $u_0\in L^2(D)$, the sequence $\{u_k\}$ defined via \cref{eq:SI2} and \cref{eq:new_iterate} converges linearly to the solution $u$ of \cref{eq:RTE} with rate $\rho$, i.e.,
    \begin{equation}
        \label{eq:main_conv}
        \normM{u - u_{k}} \leq \frac{\rho^k}{1-\rho}\normM{r_0}, \quad k\geq 0.
    \end{equation}   
\end{theorem}
In general, as shown numerically below, the theoretical bound in \cref{eq:main_conv} is too pessimistic, because it does not show the dependence on $W_{N}$. 
Indeed, by choosing $W_N=\{0\}$, we obtain $u_k=z_k$ if $u_0=z_0$, and the bound in \cref{eq:main_conv} resembles the error estimate for the iteration defined through \cref{eq:SI}.
However, we highlight the robust convergence of the proposed scheme for any choice of $W_{N}$.
For example, $W_{N}$ may contain previous iterates, which allows us to relate \cref{eq:SI2}, \cref{eq:new_iterate} to preconditioned GMRES methods or Anderson acceleration techniques, cf. \cite{WN11}. Another example is to construct $W_{N}$ by solving low-dimensional diffusion problems, see \Cref{sec:numerical_realization} for details and \Cref{sec:Numerics} for the improved convergence behavior, where we particularly consider high-order diffusion problems for highly anisotropic scattering.

The outlined approach serves as a blueprint for constructing discrete schemes. To do so, we will employ suitable (Galerkin) discretization schemes such that the monotonicity properties of the residuals are automatically guaranteed.
In view of the inversion of the transport term in \cref{eq:SI2}, we will particularly focus on discontinuous Galerkin discretizations in $\ss$ considered in \cite{DPS22}, which allow for straightforward parallelization. Such discretizations inherit similar convergence properties as the iteration described above.
This general framework can be related to existing work, as discussed next.

\subsection{Related works}
\label{sec:Related_works}

The scheme \cref{eq:SI} has been combined with several discretization methods to obtain a practical solver for radiative transfer problems. These numerical schemes are typically local in  $\ss$, such as the discrete ordinates method, also known as $S_N$-method. We refer to \cite{AL02,ML86} for an overview of classical approaches and well-established references, and to \cite{RGK12, SH20, SH21} for more recent strategies.
The main drawback of \cref{eq:SI} is the well-known slow convergence for $\rho\approx 1$. Several approaches for the acceleration of \cref{eq:SI} have been proposed in the literature, cf. \cite{AL02, ML86} for a discussion.
As observed in \cite{AL02}, \cref{eq:SI} is a preconditioned Richardson iteration for solving \cref{eq:RTE}. 
One approach to obtain faster convergence is to employ more effective preconditioners.
Among the most popular ones, we mention preconditioners that are based on solving (non-)linear diffusion problems, which are well motivated by asymptotic analysis, see again, e.g., \cite{AL02} for classical approaches, or \cite{AW18, WA18} for more recent developments.

The success of diffusion-based acceleration schemes hinges on so-called \textit{consistent} discretization of \cref{eq:RTE} and the corresponding diffusion problem \cite{AL02}. 
In \cite{RW10} consistent correction equations are obtained for two-dimensional problems with anisotropic scattering by using a modified interior penalty discontinuous Galerkin approximation for the diffusion problem. The corresponding acceleration scheme is, however, less effective for highly heterogeneous optical parameters. A discrete analysis of similar methods for high-order discontinuous Galerkin discretizations can be found in \cite{HSMT20}. We refer also to \cite{SHH20} for the development of preconditioners for heterogeneous media.
Instead of constructing special discretizations for the diffusion problems for each discretization scheme of \cref{eq:RTE}, consistent discretizations can automatically be obtained by using subspace corrections of suitable Galerkin approximations of \cref{eq:RTE} \cite{PS20, OPHY23, DPS22}. 
In \cite{OPHY23} isotropic scattering problems have been solved for a discrete ordinates-discontinuous Galerkin discretization using nonlinear diffusion problems and Anderson acceleration. The approach presented in \cite{OPHY23} reduces the full transport problem to a nonlinear diffusion equation for the angular average only, which is, however, not possible for anisotropic scattering.
The approach taken in \cite{DPS22} for anisotropic scattering problems employs a positive definite, self-adjoint second-order form of \cref{eq:RTE}, which facilitates the convergence analysis, but it requires another iterative method to actually apply the resulting matrices. The approach taken here avoids such extra inner iterations at the expense of dealing with indefinite problems.
The flexibility of our approach in constructing the spaces $W_{N}$ for the residual minimization allows us to employ similar subspaces as in \cite{DPS22}, which have been shown to converge robustly for arbitrary meshes and for forward-peaked scattering.

To treat forward-peaked scattering, for a one-dimensional radiative transfer equation, \cite{WPT15} applies nonlinear diffusion correction and Anderson acceleration, which minimizes the residual over a certain subspace, and is therefore conceptually close to our approach outlined above. Different from \cite{WPT15}, where a combination of the $S_N$-method with a finite difference method for the discretization of $\rr$ has been used and the corresponding minimizations are done in the Euclidean norm, our framework allows for general discretizations for multi-dimensional problems, such as arbitrary order (discontinuous) Galerkin schemes, to discretize $\ss$ and $\rr$.
Moreover, our framework allows us to employ higher-order correction equations, similar to \cite{DPS22}, whose effectiveness becomes apparent in our numerical examples for highly forward-peaked scattering. In addition, our Hilbert space approach, which is provably convergent, leads to algorithms that behave robustly under mesh refinements.

A second approach for accelerating \cref{eq:SI} is to replace the preconditioned Richardson iterations with other Krylov space methods.
For instance, \cite{WWM04} employs a GMRES method, which is preconditioned by solving a diffusion problem, to solve three-dimensional problems with isotropic scattering.
To treat highly forward-peaked scattering, \cite{TRM12} combines GMRES with an angular (in the $\ss$ variable) multigrid method to accelerate convergence; see also \cite{L10, KR14}, and \cite{DBSSP21} for a comparison of multilevel approaches.
By appropriately choosing $W_{N}$ in \cref{eq:minimization_problem} the approach outlined in \Cref{sec:Approach} can be related to a preconditioned GMRES method.
Since the domain $D$ has dimension $2d-1$, building up a full Krylov space during GMRES iterations becomes prohibitive in terms of memory, and GMRES has to be restarted.
Our numerical results, cf. also \cite{DPS22}, show that high-order diffusion corrections can lead to effective schemes with small memory requirements for highly forward-peaked scattering.

\subsection{Outline}
\label{sec:Outline}
The remainder of the manuscript is organized as follows.
In \Cref{sec:Notation_and_preliminaries} we introduce notation and basic assumptions on the optical parameters, and we present a weak formulation of \cref{eq:RTE}, \cref{eq:RTE_boundary_condition}, which allows for a rigorous proof of \Cref{lem:contractivity_residuals} in \Cref{sec:Contraction_properties_SI}. In \Cref{sec:Residual_minimization} we turn to the analysis of the minimization problem \cref{eq:minimization_problem} and prove \Cref{thm:main_result}.
In \Cref{sec:numerical_realization} we discuss a discretization strategy that implements the approach described in \Cref{sec:Approach} such that our main convergence results remain true for the discrete systems. We discuss several choices of $W_{N}$ there.
The practical performance of the proposed methodology for different choices of spaces $W_{N}$ in \cref{eq:minimization_problem} is investigated in \Cref{sec:Numerics}.

\section{Notation and preliminaries}
\label{sec:Notation_and_preliminaries}
In the following, we recall the main functional analytic framework and state the variational formulation of \cref{eq:RTE}-\cref{eq:RTE_boundary_condition}, with a well-posedness result. 

\subsection{Function spaces}
\label{sec:Function_spaces}

We denote with $V_0:=L^2(D)$ the usual Hilbert space of square integrable functions on the domain $D:=R\times S^{d-1}$, with inner product $\inner{\cdot}{\cdot}$ and induced norm $\normc$. In order to incorporate boundary conditions, we assume that $R$ has a Lipschitz boundary and we denote with $L^2(\d D_-; |\ss\cdot\nn|)$ the space of weighted square-integrable functions on the inflow boundary, and with $\inner{\cdot}{\cdot}_{\d D_-}$ the corresponding inner product. For smooth functions $v,w\in C^{\infty}(\bar D)$ we define
\begin{equation}
    \label{eq:inner_product_V1}
    \inner{v}{w}_{V_1} := \inner{v}{w} + \inner{\ss\cdot\nabla v}{\ss\cdot\nabla w}+ \inner{|\ss\cdot\nn|v}{w}_{\d D_-},
\end{equation}
and we denote with $V_1$ the completion of $C^{\infty}(\bar D)$ with respect to the norm associated with \cref{eq:inner_product_V1}:
\begin{equation*}
    V_1 = \left\{ v\in V_0: \ss\cdot\nabla v \in V_0,\, v|_{\d D_-} \in L^2(\d D_-;|\ss\cdot\nn|)\right\}.
\end{equation*}
For functions $v,w\in V_1$, we recall the following integration by parts formula, see, e.g., \cite{ES12},
\begin{equation}
    \label{eq:integration_by_parts}
    \inner{\ss\cdot \nabla v}{w} = - \inner{v}{\ss\cdot\nabla w} + \inner{\ss\cdot\nn v}{w}_{\d D_-}.
\end{equation}

\subsection{Optical parameters and data}
\label{sec:Optical_parameters_and_data}
Standard assumptions for the source data are $q \in V_0$ and $q_\d \in L^2(\d D_-; |\ss\cdot\nn|)$. The optical parameters $\sigma_s$ and $\sigma_t$ are supposed to be positive and essentially bounded functions of $\rr$. The medium $R$ is assumed to be absorbing, i.e., there exists $c_0>0$ such that $\sigma_a \geq c_0$ a.e. in $R$. This hypothesis ensures that the ratio between the scattering rate and the total attenuation rate is strictly less than $1$, i.e., $\rho:=\norminf{\sigma_s/\sigma_t} < 1$. 
We assume that the phase function $\theta:[-1,1]\to\RRR$ is non-negative and normalized such that $\int_{S^{d-1}}\theta(\ss\cdot\ss^\p) d\ss^\p=1$ for a.e. $\ss\in S^{d-1}$.
To ease the notation we introduce the operators $\S,\Theta:V_0\to V_0$ such that
\begin{equation*}
    (\S v)(\rr,\ss) := \sigma_s(\rr)(\Theta v)(\rr,\ss), \qquad (\Theta v)(\rr,\ss):= \int_{S^{d-1}} \theta(\ss\cdot\ss^\p) v(\rr,\ss^\p)\,d\ss^\p.  
\end{equation*}
We recall that $\S,\Theta$ are self-adjoint bounded linear operators, with operator norms bounded by $\|\sigma_s\|_\infty$ and $1$, respectively, i.e., $\Vert \Theta v \Vert_{V_0} \leq \Vert v \Vert_{V_0}$, see, e.g., \cite[Lemma 2.6]{ES12}.

\subsection{Even-odd splitting}
\label{sec:Even-odd_splitting}
For $v\in V_0$, we define its even and odd parts, identified by the superscripts $"+"$ and $"-"$, respectively, by
\begin{equation*}
    v^{\pm}(\rr,\ss) := \dfrac{1}{2} \left( v(\rr, \ss) \pm v(\rr, -\ss) \right).
\end{equation*}
Accordingly, for any space $V$ we denote by $V^{\pm} := \{ v^{\pm}:v\in V\}$ the subspaces of even and odd functions of $V$. In particular, any $V\in \{V_0, V_1\}$ has the orthogonal (with respect to the inner product of $V$) decomposition $V = V^+ \oplus V^-$. Following \cite{ES12}, a suitable space for the analysis of the radiative transfer equation is the space of mixed regularity
\begin{equation*}
    W := V_1^+ \oplus V_0^-,
\end{equation*}
where only the even components have weak directional derivatives in $V_0$.

\subsection{Variational formulation}
\label{sec:variational_formulation}
Assuming that $u$ is a smooth solution to \cref{eq:RTE}-\cref{eq:RTE_boundary_condition}, we use standard procedures to derive a weak formulation. Multiplying \cref{eq:RTE} by a smooth test function $v$, splitting the functions in their even and odd components, and using the integration by parts formula \cref{eq:integration_by_parts} to handle the term $(\ss\cdot\nabla u^-,v^+)$, we obtain the following variational principle, see \cite{ES12} for details: find $u\in W$ such that for all $v\in W$
\begin{equation}
    \label{eq:weak_formulation}
    t(u,v) = s(u,v) + \ell(v),
\end{equation}
with bilinear forms $t,s:W\times W \to \RRR$, and linear form $\ell:W\to\RRR$ defined by
\begin{align}
    \label{eq:bilinear_form_T}
    t(u,v) & := \inner{\ss\cdot\nabla u^+}{v^-} - \inner{u^-}{\ss\cdot\nabla v^+} + \inner{\sigma_t u}{v} + (|\ss\cdot\nn|\, u^+,v^+)_{\d D}, \\
    s(u,v) & := \inner{\S u}{v}, \\
    \ell(v) & := \inner{q}{v} - 2 (\ss\cdot\nn\, q_\d, v^+)_{\d D_-}.
\end{align}
As shown in \cite[Section 3]{ES12} the assumptions imposed in \Cref{sec:Optical_parameters_and_data} imply that there exists a unique solution of \cref{eq:weak_formulation} satisfying 
\begin{equation*}
    \norm{u}_W \leq C(\norm{q} + \norm{q_\d}_{L^2(\d D_-;|\ss\cdot\nn|)}),
\end{equation*}
with constant $C>0$ depending only on $c_0$ and $\norminf{\sigma_t}$. 
Let us also recall from \cite{ES12} that the odd part of the weak solution $u\in W$ enjoys $u^- \in V_1$, and that $u$ satisfies \cref{eq:RTE} almost everywhere, and \cref{eq:RTE_boundary_condition} in the sense of traces.
Before proceeding, we collect some properties of the bilinear forms $s$ and $t$, which we will use in our analysis below.
\begin{lemma}
\label{lem:properties_s_t}
    Let $\rho = \|\sigma_s/\sigma_t\|_\infty$. For any $u,v\in W$ the following inequalities hold:
    \begin{align}
        \normM{u}^2 &\leq t(u,u),\label{eq:pos_t}\\
        s(u,v) &\leq \rho \normM{u} \normM{v}.    \label{eq:cont_s}
    \end{align}
\end{lemma}
\begin{proof} 
    The inequality \cref{eq:pos_t} follows immediately from \cref{eq:bilinear_form_T} and the definition of the weighted norm. To show \cref{eq:cont_s} we apply the Cauchy-Schwarz inequality to obtain that
    \begin{align*}
        s(u,v) = \inner{\frac{\sigma_s}{\sigma_t}  \Theta(\sqrt{\sigma_t}u)}{\sqrt{\sigma_t}v} 
        \leq \rho \|\Theta(\sqrt{\sigma_t} u)\|  \|\sqrt{\sigma_t} v\| \leq \rho \normM{u} \normM{v},
    \end{align*}
    where we used the boundedness of $\Theta$ in the last step.
\end{proof}

\section{Contraction properties of the source iteration}
\label{sec:Contraction_properties_SI}
Equipped with the notation from the previous section, we can write \cref{eq:SI2} as follows: Given $u_k\in W$, compute $u_{k+1/2}\in W$ such that
\begin{equation}
    \label{eq:iteration_equation_bilinear}
    t(u_{k+1/2}, v) = s(u_k, v) + \ell(v), \quad \forall v\in W.
\end{equation}
The following result is well-known, and we provide a proof for later reference.
\begin{lemma}
    \label{lem:contractivity_errors}
    Let $u_k, u_{k+1/2}\in W$ be related via \cref{eq:iteration_equation_bilinear}. Then it holds that
    \begin{equation*}        
        \normM{u-u_{k+1/2}} \leq \rho\normM{u-u_k},
    \end{equation*}
    with $\rho = \|\sigma_s/\sigma_t\|_\infty$.
\end{lemma}
\begin{proof}
    Abbreviating $e_{k+1/2}:=u-u_{k+1/2}$ and $e_k:=u-u_k$, \cref{eq:weak_formulation} and \cref{eq:iteration_equation_bilinear} imply that
    \begin{equation}
        \label{eq:half_step_error_equation}
        t(e_{k+1/2},v) = s(e_k, v), \quad \forall v\in W.
    \end{equation}
    Therefore, setting $v=e_{k+1/2}$ in \cref{eq:half_step_error_equation}, using \cref{eq:pos_t} and \cref{eq:cont_s}, we obtain the estimates
    \begin{align*}
        \normM{e_{k+1/2}}^2 \leq t(e_{k+1/2},e_{k+1/2}) = s(e_k, e_{k+1/2}) \leq \rho \normM{e_{k+1/2}} \normM{e_k},
    \end{align*}
    which concludes the proof.
\end{proof}
Along the lines of \cref{eq:residual} and \cref{eq:preconditioned_residual}, we define the residual operator $\tilde{\R}:W\to W^{\ast}$ by
\begin{equation*}
    \duality{\tilde{\R}(w)}{v} := \ell(v) - (t(w,v) - s(w,v)), \quad \forall v\in W.
\end{equation*}
Here, $W^*$ denotes the dual space of $W$, and $\langle\cdot,\cdot\rangle$ the corresponding duality pairing. The dual norm is defined by $\|\ell\|_{W^*}:=\sup_{\|v\|_W=1}\ell(v)$.
The preconditioned residual operator $\R:W\to W$ is defined by solving the following transport problem without scattering,
\begin{equation*}
    t(\R(w), v) = \duality{\tilde{\R}(w)}{v}, \quad \forall v \in W.
\end{equation*}
Using the arguments used to analyze \cref{eq:weak_formulation} in \Cref{sec:variational_formulation}, one  shows that the operator $\R$ is well-defined.
The following result ensures that the corresponding preconditioned residuals decay monotonically, which verifies \Cref{lem:contractivity_residuals}.
\begin{lemma}
    \label{lem:contractivity_residuals_operator}
    Let $u_k, u_{k+1/2}\in W$ be related via \cref{eq:iteration_equation_bilinear}. Then it holds that
    \begin{align*}
        \normM{\R(u_{k+1/2})}\leq \rho \normM{\R(u_k)}.
    \end{align*}
\end{lemma}
\begin{proof}
    Denoting $r_k=\R(u_k)$ and using \cref{eq:half_step_error_equation}, we observe that
    \begin{align*}
        t(r_k, v) & = \duality{\tilde{\R}(u_k)}{v} = \ell(v) - (t(u_k, v) - s(u_k, v))\\
        & = t(e_k,v) - s(e_k, v) = t(e_k,v) - t(e_{k+1/2}, v) \\ & = t(u_{k+1/2}-u_k, v),
    \end{align*}
    for all $v\in W$, i.e., $r_k=u_{k+1/2}-u_k$.
    Using a similar argument and \cref{eq:half_step_error_equation}, we further obtain that
    \begin{align*}
        t(r_{k+1/2}, v) & = \duality{\tilde{\R}(u_{k+1/2})}{v} = \ell(v) - (t(u_{k+1/2}, v) - s(u_{k+1/2}, v) ) \\
        & = t(e_{k+1/2},v) - s(e_{k+1/2}, v) = s(e_k,v) - s(e_{k+1/2}, v) \\ & = s(u_{k+1/2} - u_k, v) = s(r_k, v),
    \end{align*}
    for all $v\in W$. In view of the proof of \Cref{lem:contractivity_errors}, the proof is complete.
\end{proof}
The following result allows us to relate the error to the residual quantitatively.
\begin{lemma}
    \label{lem:upper_bound_error_residual}
    Let $u\in W$ be the solution to \cref{eq:weak_formulation}. Then the following estimate holds
    \begin{align*}
        \normM{u-w} \leq \frac{1}{1-\rho} \normM{\R(w)}\qquad\forall w\in W.
    \end{align*}
\end{lemma}
\begin{proof}
    Let us introduce the linear operator $\R_0:W\to W$ defined by $\R_0 w:=\R(w)-L$, where $L\in W$ is defined by the relation $t(L,v)=\ell(v)$, for $v\in W$.
    We observe that for $v\in W$, inequality \cref{eq:pos_t}, the definition of $\R$ and \cref{eq:cont_s} imply that
    \begin{align*}
        \normM{\R_0 v + v}^2 \leq t( \R_0 v + v,\R_0v + v) = s(v,\R_0 v +v)\leq \rho\normM{v}\normM{\R_0 v +v}.
    \end{align*}
    The triangle inequality thus implies that $\normM{v} \leq \normM{v+\R_0v} + \normM{\R_0v} \leq \rho\normM{v}+ \normM{\R_0 v}$, i.e., 
    \begin{align}
        \label{eq:equivalence_res_error}
        \normM{v}\leq \normM{\R_0 v}/(1-\rho)\qquad \forall v\in W.
    \end{align}
    The assertion follows from \cref{eq:equivalence_res_error} with $v=u-w$ and the observation that $\R_0(u-w)=-\R(w)$.
\end{proof}

\section{Residual minimization}
\label{sec:Residual_minimization}
In view of \Cref{lem:upper_bound_error_residual} the difference between the weak solution $u$ of \cref{eq:weak_formulation} and any element $w\in W$ is bounded from above by the norm of the residual $\R(w)$ associated with $w$. Hence, if we can construct $w$ such that $\R(w)=0$, then $u=w$.
\Cref{lem:contractivity_residuals_operator} shows that the half-step \cref{eq:half_step_error_equation} reduces the norm of the residual by the factor $\rho$.
These observations motivate us to modify $u_{k+1/2}$ such that the corresponding residual becomes smaller. 
In order to obtain a feasible minimization problem, let $W_{N}\subset W$ be a subspace of finite dimension $N$.
We then compute the modification $u^c_{k+1/2}\in W_{N}$ such that
\begin{align}
    \label{eq:minimization}
    u^c_{k+1/2} := \underset{w \in W_{N}}{\mathrm{argmin}} \normM{\R(u_{k+1/2} + w)}^2.
\end{align}
The new iterate of the scheme is then defined as
\begin{align}
    \label{eq:update}
    u_{k+1} \colonequals u_{k+1/2}+ u^c_{k+1/2},
\end{align}
and the procedure can restart. 

\begin{lemma}
    \label{lem:minimization}
    The minimization problem in \cref{eq:minimization} has a unique solution $u_{k+1/2}^c\in W_{N}$.
\end{lemma}
\begin{proof}
    Since $\R(u_{k+1/2} + w)=\R(u_{k+1/2}) + \R_0 w$ with $\R_0$ as defined in the proof of \Cref{lem:upper_bound_error_residual}, the optimization problem \cref{eq:minimization} is a least-squares problem for the norm $\normM{\cdot}^2$ with linear operator $\R_0$ and data $-\R(u_{k+1/2})$. A necessary condition for a minimizer $w^*\in W_{N}$ is therefore
    \begin{align}
        \label{eq:necessary}
        (\R_0 w^*,\R_0 v)_{\sigma_t} = -(\R(u_{k+1/2}),\R_0 v)_{\sigma_t} \quad\forall v\in W_{N}.
    \end{align}
    In view of \cref{eq:equivalence_res_error}, the minimizer $u_{k+1/2}^c=w^*$ is unique, and hence exists, because $W_{N}$ is finite dimensional and \cref{eq:necessary} is a quadratic problem.
\end{proof}
\textbf{Proof of \Cref{thm:main_result}.}
Using \Cref{lem:upper_bound_error_residual},
the minimization property of $u_{k+1}$, and \Cref{lem:contractivity_residuals_operator} we obtain that
\begin{align*}
    (1-\rho)\normM{e_{k+1}}\leq \normM{\R(u_{k+1})}\leq \rho^{k+1}\normM{\R(u_0)},
\end{align*}
which concludes the proof.

\section{Numerical realization}
\label{sec:numerical_realization}
The convergent iteration in infinite-dimensional Hilbert spaces described in the previous sections serves as a blueprint for constructing numerical methods.
The variational character of the scheme allows translating the infinite-dimensional iteration directly to a corresponding convergent iteration in finite-dimensional approximation spaces $\Wh$ of $W$.

\subsection{Galerkin approximation}
To be specific, we employ for $\Wh$ a construction as in \cite{DPS22}. Let $\TRh$ and $\TSh$ be shape regular, quasi-uniform and conforming triangulations of $R$ and $S$, respectively. Here, $h>0$ denotes a mesh-size parameter. In addition, we require that $-K^S\in\TSh$ for any $K^S\in \TSh$ in order to be able to properly handle even and odd functions.
We then denote by $\Sh^\pm$ the corresponding finite element spaces of even piecewise constant (+) and odd piecewise linear (--) functions associated with $\TSh$.
Similarly, we denote by $\Xh^\pm$ the finite element spaces consisting of piecewise constant (--) and continuous piecewise linear (+) functions associated with $\TRh$.
Please note that we use the symbol $\pm$ in $\Xh^\pm$ for notational convenience and not to indicate whether a function of the spatial variable is even or odd.
Our considered approximation space is then defined as $\Wh=\Sh^+\otimes\Xh^+ + \Sh^-\otimes\Xh^-$.
The Galerkin approximation of \cref{eq:weak_formulation} reads: Find $u_h\in \Wh$ such that
\begin{equation}
    \label{eq:weak_formulation_discrete}
    t(u_h,v_h) = s(u_h,v_h) + \ell(v_h)\quad\forall v_h\in \Wh.
\end{equation}
As shown in \cite{ES12}, \cref{eq:weak_formulation_discrete} has a unique solution $u_h\in \Wh$ that is uniformly (in $h$) bounded by $\|\ell\|_{W^*}$. Moreover, there is a constant $C>0$ independent of the discretization parameters such that
\begin{align*}
    \|u-u_h\|_W \leq C \inf_{v_h\in\Wh} \|u-v_h\|_{W},
\end{align*}
i.e., $u_h$ is a quasi-best approximation to $u$ in $\Wh$.

\subsection{Iterative scheme}
The discretization of \cref{eq:iteration_equation_bilinear}, \cref{eq:update} becomes:
Given $u_{h,k}\in \Wh$, compute $u_{h,k+1/2}\in \Wh$ such that
\begin{equation}
    \label{eq:iteration_equation_bilinear_discrete}
    t(u_{h,k+1/2}, v_h) = s(u_{h,k}, v_h) + \ell(v_h), \quad \forall v_h\in \Wh.
\end{equation}
The discretization of the preconditioned residual operator is $\Rh:\Wh\to\Wh$ defined via
\begin{equation*}
    t(\R_h(w_h), v_h) =\ell(v_h)-(t(w_h,v_h)-s(w_h,v_h)), \quad \forall v_h \in \Wh,
\end{equation*}
and the corresponding corrections are computed via the minimization problem
\begin{align}
    \label{eq:minimization_discrete}
    u^c_{h,k+1/2} := \underset{w_h \in W_{h,N}}{\mathrm{argmin}} \normM{\Rh(u_{h,k+1/2} + w_h)}^2,
\end{align}
where $\WhN\subset \Wh$.
The new iterate of the discrete scheme is then defined accordingly by
\begin{align}
    \label{eq:update_discrete}
    u_{h,k+1} \colonequals u_{h,k+1/2}+ u^c_{h,k+1/2}.
\end{align}
Repeating the arguments of \Cref{sec:Contraction_properties_SI} and \Cref{sec:Residual_minimization}, we obtain the following convergence statement.
\begin{theorem}
    \label{thm:convergence_discrete}
    For any $u_{h,0}\in\Wh$, the sequence $\{u_{h,k}\}$ defined by \cref{eq:iteration_equation_bilinear_discrete}, \cref{eq:update_discrete} converges linearly to the solution $u_h$ of \cref{eq:weak_formulation_discrete}, i.e.,
    \begin{align*}
        \normM{u_h-u_{h,k}} \leq \frac{\rho^k}{1-\rho}\normM{\Rh(u_{h,0})}.
    \end{align*}
    Moreover, the residuals converge monotonically, i.e., $\normM{\Rh(u_{h,k+1})}\leq \rho\normM{\Rh(u_{h,k})}$.
\end{theorem}

\subsection{Formulation in terms of matrices}
Choosing basis functions for $\Sh^\pm$ and $\Xh^\pm$ allows us to rewrite the iteration $u_{h,k}\mapsto u_{h,k+1}$ in corresponding coordinates.
Denote $\{\varphi_i\}_{i=1}^{\nre}$ and $\{\chi_j\}_{j=1}^{\nro}$ the usual basis functions with local support of $\Xh^+$ and $\Xh^-$, i.e., $\varphi_i$ vanishes in all vertices of $\TRh$ except in the $i$th one, while $\chi_j$ vanishes in all elements of $\TRh$ except in the $j$th one. Similarly, we denote $\{\mu_k\}$ the basis of $\Sh^+$ such that $\mu_k$ vanishes in all elements of $\TSh$ except in $K^S_k$ and $-K^S_k$. Eventually, we denote by $\{\psi_l\}_{l=1}^{\nso}$ the basis of $\Sh^-$ such that $\psi_l$ vanishes in all vertices belonging to $\TSh$ except for the vertices $p_l$ and $-p_l$.
We may then write the even and odd parts of $u_h$ as 
\begin{align*}
    u_h^+ = \sum_{i=1}^{\nre}\sum_{{k=1}}^{\nse} \uu_{i,k}^+ \varphi_i \mu_k,\qquad 
    u_h^- = \sum_{j=1}^{\nro}\sum_{{l=1}}^{\nso} \uu_{j,l}^- \chi_j \psi_k,
\end{align*}
and \cref{eq:weak_formulation_discrete} turns into the linear system
\begin{align}
\label{eq:RTE_discrete}
	\TT\uu=\SS\uu +\ll
\end{align}
with matrices $\TT$ and $\SS$ having the following block structure,
\begin{align}
    \TT := \begin{bmatrix} \BB +\MMe & -\AA^T \\ \AA & \MMo\end{bmatrix},\qquad
	\SS:=\begin{bmatrix}\SSe &\\& \SSo\end{bmatrix},\qquad \ll:=\begin{bmatrix}\ll^+ \\ \ll^-\end{bmatrix}.
\end{align}
The individual blocks are given as follows:
\begin{alignat*}{6}
    \SSe &:= \bThetae\otimes \bMsce, &&\qquad& \SSo &:= \bThetao\otimes \bMsco,\\
    \MMe &:= \bMSpheree\otimes \bMtre, &&\qquad& \MMo &:= \bMSphereo\otimes \bMtro,\\
    \AA &:= \sum_{i=1}^d \bASphere_i\otimes\DD_i,&&\qquad& \BB &:= {\rm blkdiag}(\BB_1,\ldots,\BB_{\nse}),
\end{alignat*}
with matrices
\begin{alignat*}{6}
   (\bMsce)_{i,i'} &:= \int_R \sigma_t \varphi_i \varphi_{i'}\,d\rr,
   &&\qquad& (\bThetae)_{k,k'} &:= \int_S\int_S \theta(\ss\cdot\ss') \mu_k(\ss') \mu_{k'}(\ss)\,d\ss'\,d\ss,\\
    (\bMsco)_{j,j'} &:= \int_R \sigma_t \chi_j \chi_{j'}\,d\rr,
    &&\qquad& (\bThetao)_{l,l'} &:= \int_S\int_S \theta(\ss\cdot\ss') \psi_l(\ss') \psi_{l'}(\ss)\,d\ss'\,d\ss,\\
    (\DD_n)_{i,k} &:= \int_R \frac{\d \varphi_i}{\d r_n} \chi_k\,d\rr, &&\qquad&
    (\bASphere_i)_{l,k} &:= \int_S \ss_i \psi_l \mu_k\,d\ss,
    \\
    (\BB_k)_{i,i'}&:=\int_{\d R} \varphi_i\varphi_{i'}\omega_k\,d\rr, &&\qquad& \omega_k &:= \int_S |\ss\cdot\nn| (\mu_k)^2\,d\ss,\\
    (\bMSpheree)_{k,k'}&:=\int_S \mu_k \mu_{k'}\,d\ss, &&\qquad&(\bMSphereo)_{l,l'}&:=\int_S \psi_l \psi_{l'}\,d\ss.
\end{alignat*}
The vectors $\ll^\pm$ are obtained from inserting basis functions into the linear functional $\ell$.
We mention that all matrices are sparse, except $\bThetae$ and $\bThetao$, which can be applied efficiently using hierarchical matrix compression, see \cite{DPS22} -- for moderate $\nse$, $\nso$ dense linear algebra is efficient, too.
In particular, the matrices $\bMsco$ and $\bMSphereo$ are diagonal and $3\times 3$ block diagonal, respectively, i.e., $\MMo$ can be inverted efficiently. 
Using these matrices, \cref{eq:iteration_equation_bilinear_discrete} turns into the linear system
\begin{align}
    \label{eq:iteration_equation_bilinear_matrix}
    \TT \uu_{k+1/2} = \SS \uu_k + \ll,
\end{align}
which can be solved as described in \Cref{rem:solve_system}.
Denote $\RR(\ww)$ the coordinate vector of $\Rh(w)$. Then $\RR(\ww)$ is determined by solving
\begin{align*}
	\TT\RR(\ww)=\ll-(\TT-\SS)\ww.
\end{align*}
Hence, the operator ${\Rh}_0$ discretizing $\R_0$ becomes $\RR_0\ww=\RR(\ww)-\RR(0)$.
The update is computed via the minimization \cref{eq:minimization_discrete}, which becomes, cf. \cref{eq:necessary},
\begin{align}
    \label{eq:necessary_discrete}
    (\RR_0^N)^T \MM \RR_0^N \ww^* = - (\RR_0^N)^T \MM \RR(\uu_{k+1/2}),
\end{align}
where $\MM={\rm blkdiag}(\MM^+,\MM^-)$.
Here, $\RR_0^N=\RR_0\WW_{h,N}$, where $\WW_{h,N}$ is a matrix, whose columns correspond to the coordinates of a basis for $\WhN$.
Depending on the conditioning of the matrix $\WW_{h,N}$, the system in \cref{eq:necessary_discrete} might be ill-conditioned. To stabilize the solution process, we compute the minimum-norm solution.
The coordinate vector for the correction $u_{h,k+1/2}^c$ is then given by  $\uu_{k+1/2}^c = \WW_{h,N} \ww^*$, which gives the following update formula for the coordinates of the new iterate
\begin{align}
    \label{eq:new_iterate_discrete}
    \uu_{k+1}= \uu_{k+1/2} + \uu_{k+1/2}^c.
\end{align}
The residual corresponding to $u_{h,k+1}$ can be updated according to
$\rr_{k+1}=\RR \uu_{k+1/2}+ \RR_0^N \ww^*$.
\Cref{thm:convergence_discrete} ensures that $\uu_{k}$ converges linearly to the solution $\uu$ of \cref{eq:RTE_discrete}.

\begin{remark}
\label{rem:solve_system}
Solving for $\uu_{k+1/2}$ can be done in two steps. First, one may solve the symmetric positive definite system
\begin{align}
    \label{eq:even_parity}
    (\AA^T (\MM^-)^{-1}\AA + \MM^+ + \BB) \uu_{k+1/2}^+ = \SSe \uu_{k}^+ + \ll^+ + \AA^T (\MM^-)^{-1}( \ll^- + \SSo \uu_k^{-}).
\end{align}
The system in \cref{eq:even_parity} is block diagonal with $\nse$ many sparse blocks of size $\nre\times\nre$ and can be solved in parallel, with straightforward parallelization over each element of $\TSh$.
Second, one may then retrieve the odd part by solving the system
\begin{align*}
    \MM^- \uu_{k+1/2}^- = \SSo \uu_{k}^+ - \AA \uu_{k+1/2}^+ +\ll^-,
\end{align*}
which can be accomplished with linear complexity due to the structure of $\MMo$.
\end{remark}
\begin{remark}\label{rem:cost_minimization}
    Assuming that the dimension of $W_{h,N}$ is $N$, \cref{eq:necessary_discrete} is a dense $N\times N$ system in general, which can be solved at negligible cost for small $N$. The assembly of the corresponding linear system, i.e., of the matrix $\RR_{0}^N$ and the right-hand side $\RR(\uu_{k+1/2})=\RR_0(\uu_{k+1/2})+\RR(0)$, requires $N+1$ applications of $\RR_0$, which can be carried out as described in \Cref{rem:solve_system}.
    Since we can compute $\uu_{k+3/2}=\uu_{k+1}+\rr_{k+1}$ we can skip the solution of a linear system in \cref{eq:iteration_equation_bilinear_matrix} for $k>0$. 
    Hence, the computational cost for performing one step of the proposed method \cref{eq:iteration_equation_bilinear_matrix}, \cref{eq:new_iterate_discrete} is comparable to $N+1$ steps of a corresponding discretization of the unmodified iteration \cref{eq:SI}, plus the computational cost for setting up $W_{h,N}$. 
    Thus, assuming that the cost for setting up $W_{h,N}$ is minor, the extra cost for the minimization by solving \cref{eq:necessary_discrete} is justified as long as the iteration \cref{eq:iteration_equation_bilinear_matrix}, \cref{eq:new_iterate_discrete} converges faster than $\rho^{N+1}$.
\end{remark}

\subsection{Choice of subspaces}
The previous consideration did not depend on a particular choice of the minimization space $W_{h,N}$.
As discussed in \Cref{sec:Related_works}, a common choice for the considered radiative transfer problem is to use corrections derived from asymptotic analysis, i.e., related diffusion problems.
As has been observed in \cite{PS20,DPS22}, the corresponding discretizations of such diffusion problems can be understood as projections on constant functions in $\ss$. Such functions, in turn, can be interpreted as eigenvectors of the matrix $\bThetae$. 
Indeed, the spherical harmonics are the eigenfunctions of the integral operator in \cref{eq:RTE} and the lowest-order spherical harmonic is constant.

\subsubsection{\texorpdfstring{Constructing $W_{h,N}$ using eigenfunctions of $\bThetae$}{Constructing WhN using eigenfunctions of Thetae}}
Let $\HH_{k}\neq 0$ solve the generalized eigenvalue problem
\begin{align*}
    \bThetae \HH_k^+ = \gamma_{k} \bMSpheree \HH_k^+,
\end{align*}
for some $\gamma_k\geq 0$, where we suppose that the eigenvalues are ordered non-increasingly, i.e., $\gamma_{k}\geq \gamma_{k+1}$. We denote $H_{h,k}^+\in \Sh^+$ the corresponding (even) eigenfunctions, and define the space
\begin{align*}
    H_{K}^+ := \mathrm{span}\{ H_{h,k}^+: 1\leq k\leq K\}.
\end{align*}
We further define $H_{h,k,i}^-\in \Sh^-$ by the relation
\begin{align*}
    (H_{h,k,i}^- ,\psi_l) =(\ss_i H_{h,k}^+,\psi_l)\quad\text{for all } l=1,\ldots,\nso,\quad 1\leq i\leq d,
\end{align*}
and $H_K^- =\mathrm{span}\{H_{h,k,i}^-: 1\leq k\leq K, \ 1\leq i\leq d\}$.
Using these definitions, we can define the space 
\begin{align}
    \label{def:YK}
    Y_{h,K} := H_K^+\otimes \Xh^+ + H_K^-\otimes \Xh^-\subset W_h.
\end{align}
To derive a correction equation, we rewrite \cref{eq:half_step_error_equation} as follows
\begin{align}
    \label{eq:half_step_error_equation_discrete}
    t( u_h -u_{h,k+1/2},v_h) = s( u_h -u_{h,k+1/2},v_h)+ s(u_{h,k+1/2}-u_{h,k},v_h)\quad\forall v_h\in \Wh.
\end{align}
Similar to \cite{PS20,DPS22}, but see also \cite{AL02}, we may expect to obtain a good approximation to the error $u_h-u_{h,k+1/2}$ by solving \cref{eq:half_step_error_equation_discrete} on the subspace $Y_{h,K}$, for a certain $K$. We hence define $u_{h,k+1/2}^c\in Y_{h,K}$ as the unique solution of
\begin{align}
    \label{eq:corr_half_step_error_equation_discrete}
    t(u_{h,k+1/2}^c,v_h) = s( u_{h,k+1/2}^c,v_h)+ s(u_{h,k+1/2}-u_{h,k},v_h)\quad\forall v_h\in Y_{h,K}.
\end{align}
By construction, the space $Y_{h,K}$ satisfies the compatibility condition $\ss\cdot\nabla y_h^+\in Y_{h,K}^-=H_K^-\otimes \Xh^-$ for any $y_h^+\in Y_{h,K}^+$. Therefore, \cref{eq:corr_half_step_error_equation_discrete} has a unique solution \cite{ES12}.
If $K$ is moderately small, \cref{eq:corr_half_step_error_equation_discrete} can be solved efficiently, see \cite{DPS22} for a discussion.
We then define
\begin{align}
    \label{eq:W1}
    W_{h,N}^c := \mathrm{span}\{ u_{h,k+1/2}^c\} \subset \Wh.
\end{align}
The dimension of $W_{h,N}^c$ is $N=1$, and \cref{eq:minimization_discrete} can be carried out efficiently.
\begin{remark}
    Since \cref{eq:half_step_error_equation_discrete} is a saddle-point problem, the Galerkin projection \cref{eq:corr_half_step_error_equation_discrete} may enlarge the error. This is in contrast to \cite{DPS22}, where \cref{eq:weak_formulation_discrete} was reformulated to a second-order form, which is symmetric and positive definite. In the latter situation, Galerkin projections correspond to best-approximations in the energy norm, and therefore do not enlarge the error.
\end{remark}

\subsubsection{Enriched space}
By construction, the correction $u_{h,k+1/2}^c$ computed via solving the projected problem \cref{eq:corr_half_step_error_equation_discrete} does, in general, not satisfy \cref{eq:half_step_error_equation_discrete}.
We will also consider enriched versions of $W_{h,N}^c$ as follows.
First, we may use the corrected even iterate to find $\tsup[1]{u}_{h,k+1/2}^{c,-}\in \Wh^-$ by solving
\begin{align*}
    \MMo \big( \uu_{k+1/2}^- + \tsup[1]{\uu}_{k+1/2}^{c,-}\big)= \SSo \uu_{k+1/2}^- + \ll^- - \AA (\uu_{k+1/2}^+ +\uu_{k+1/2}^{c,+}).
\end{align*}
The (block-)diagonal structure of $\bMSphereo$ and $\bMtro$ allows us to invert $\MMo= \bMSphereo\otimes \bMtro$ efficiently.
Second, we may compute another correction as follows. Suppose that $u_{h,k+1/2}^+ +u_{h,k+1/2}^{c,+}$ is close to the even-part $u_h^+$ of the solution to \cref{eq:weak_formulation_discrete}, then, for consistency reasons, we may expect that the solution $u_{h,k+1/2}^{-}+\tsup[2]{u}_{h,k+1/2}^{c,-}$ to the following system is a good approximation to $u_h^-$:
\begin{align*}
    \MMo \big( \uu_{k+1/2}^- + \tsup[2]{\uu}_{k+1/2}^{c,-} \big) = \SSo \big( \uu_{k+1/2}^-+ \tsup[2]{\uu}_{k+1/2}^{c,-}\big)  + \ll^- - \AA (\uu_{k+1/2}^+ +\uu_{k+1/2}^{c,+}).
\end{align*}
Computing $\tsup[2]{\uu}_{k+1/2}^{c,-}$ requires the inversion of $\MMo-\SSo$, which can be accomplished using a preconditioned conjugate gradient method as done in \cite{DPS22}.
We then define the enriched space
\begin{align}
    \label{eq:tW1}
    \tsup[1]{W}_{h,N}^c := \mathrm{span}\{ u_{h,k+1/2}^{c,+},u_{h,k+1/2}^{c,-},\tsup[1]{u}_{h,k+1/2}^{c,-},\tsup[2]{u}_{h,k+1/2}^{c,-}\} \subset \Wh,
\end{align}
which can be employed in the minimization \cref{eq:minimization_discrete}. The dimension of this space is $N=4$.  

\subsubsection{Another enriched space: including previous iterates}
\label{sec:Enriched_space+AA}
Since the minimization procedure is flexible in defining the correction space $W_{h,N}$, we may not only rely on minimizing the residual over functions obtained from Galerkin subspace projection. Borrowing ideas from GMRES, given an iterate $u_{h,k}$, we will also consider the space
\begin{align}
    \label{eq:WN_def}
    \tsup[1]{W}_{h,N}^{c,m}:= \tsup[1]{W}_{h,N}^c + \mathrm{span}\{u_{h,j}: k-m\leq j\leq k\},
\end{align}
where also the previous $m$ iterates are taken into account to construct the space for minimization. It is clear that if $m\geq k$ all previous iterates are considered.
In practice, memory limitations usually require keeping $m$ small.

\section{Numerical experiments}
\label{sec:Numerics}
We will investigate the behavior of the iteration and the influence of the different subspaces for the residual minimization discussed in \Cref{sec:numerical_realization} by means of a checkerboard test problem \cite{B05}.
Here, the spatial domain is given by $R=(0,7)\times(0,7)$, the inflow boundary condition is given by $q_\d=0$, and the internal source term $q$ as well as the scattering and absorption parameter, $\sigma_s$ and $\sigma_a$, respectively, are defined in \Cref{fig:checkerboard}.
Hence, the theoretical rate of \Cref{thm:main_result} is given by $\rho=\|\sigma_s/\sigma_t\|_\infty = 0.999$ here. We consider the Henyey-Greenstein scattering phase-function with anisotropy factor $0 \leq g < 1$, i.e.,
\begin{equation*}
    \theta(\ss\cdot \ss^{\p}) := \dfrac{1}{4\pi} \dfrac{1-g^2}{[1-2g(\ss\cdot \ss^{\p}) +g^2]^{3/2}}.
\end{equation*}

If not stated otherwise, the domain $R$ is triangulated using $100\,352$ elements, i.e., $\nre=50\,625$ and $\nro=100\,352$. Moreover, we will employ $1024$ elements on the half sphere, i.e., $\nse=1024$ and $\nso=3072$, which results in $360\,121\,344$ degrees of freedom.
Spherical integration is performed by using a high-order numerical quadrature. 
We note that the accurate assembly of $\bThetae$ and $\bThetao$ may require modified integration rules for $g=0.99$. We chose to show results for this value of $g$ to verify the robustness of our approach.
For a sketch of a corresponding polyhedral approximation of the sphere see \Cref{fig:checkerboard}. 
We ran our code on a dual AMD EPYC 7742 64-Core Processor, i.e., $128$ physical cores, with 1024 GB memory and Matlab R2023b with $128$ workers. The code is freely available at \cite{CODE}.
The iterations are stopped as soon as $\normM{\R_h(u_k)}<10^{-6}$.

\begin{figure}[ht]
	\includegraphics[width=0.49\textwidth]{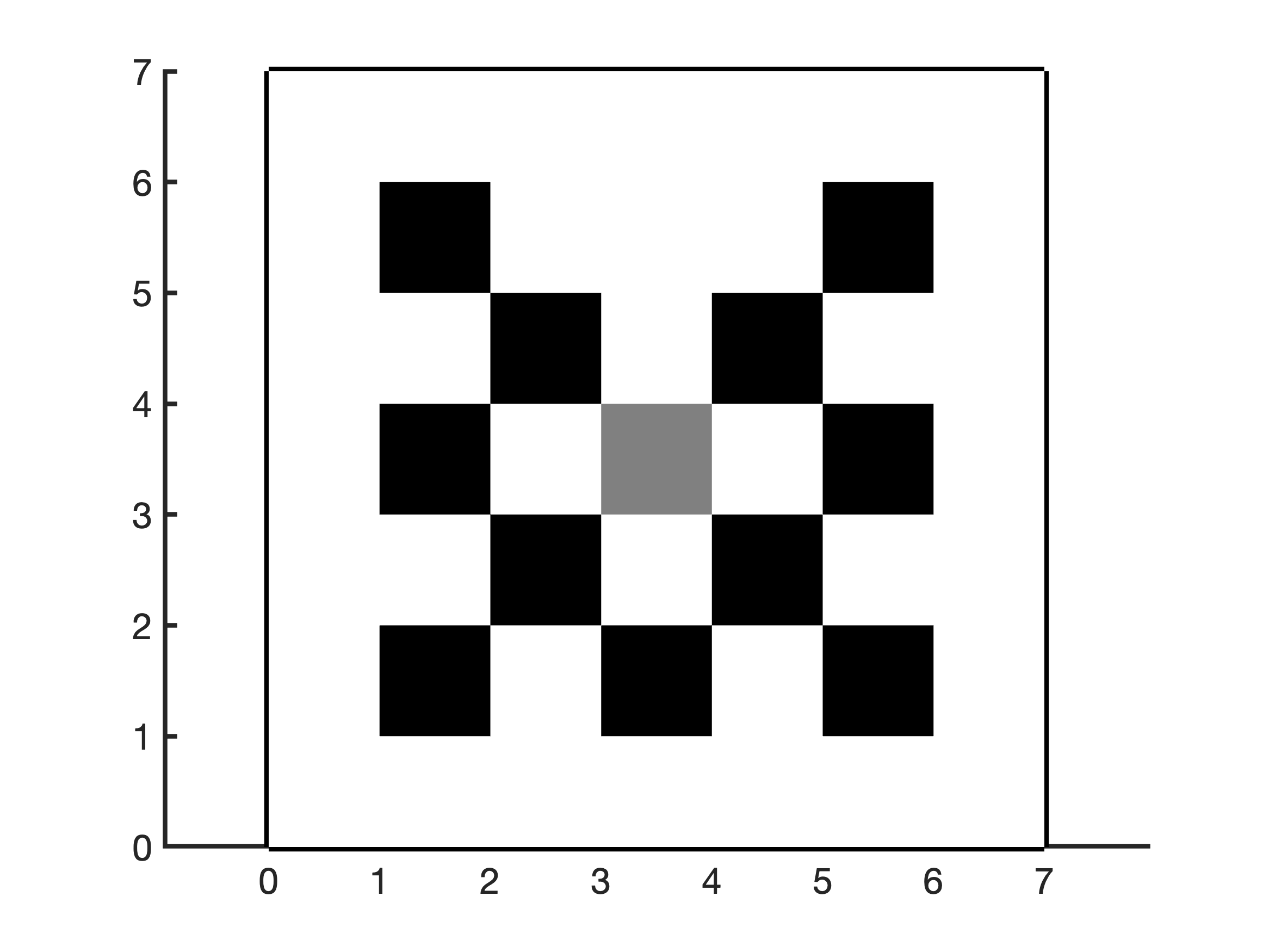}
	\includegraphics[width=0.49\textwidth]{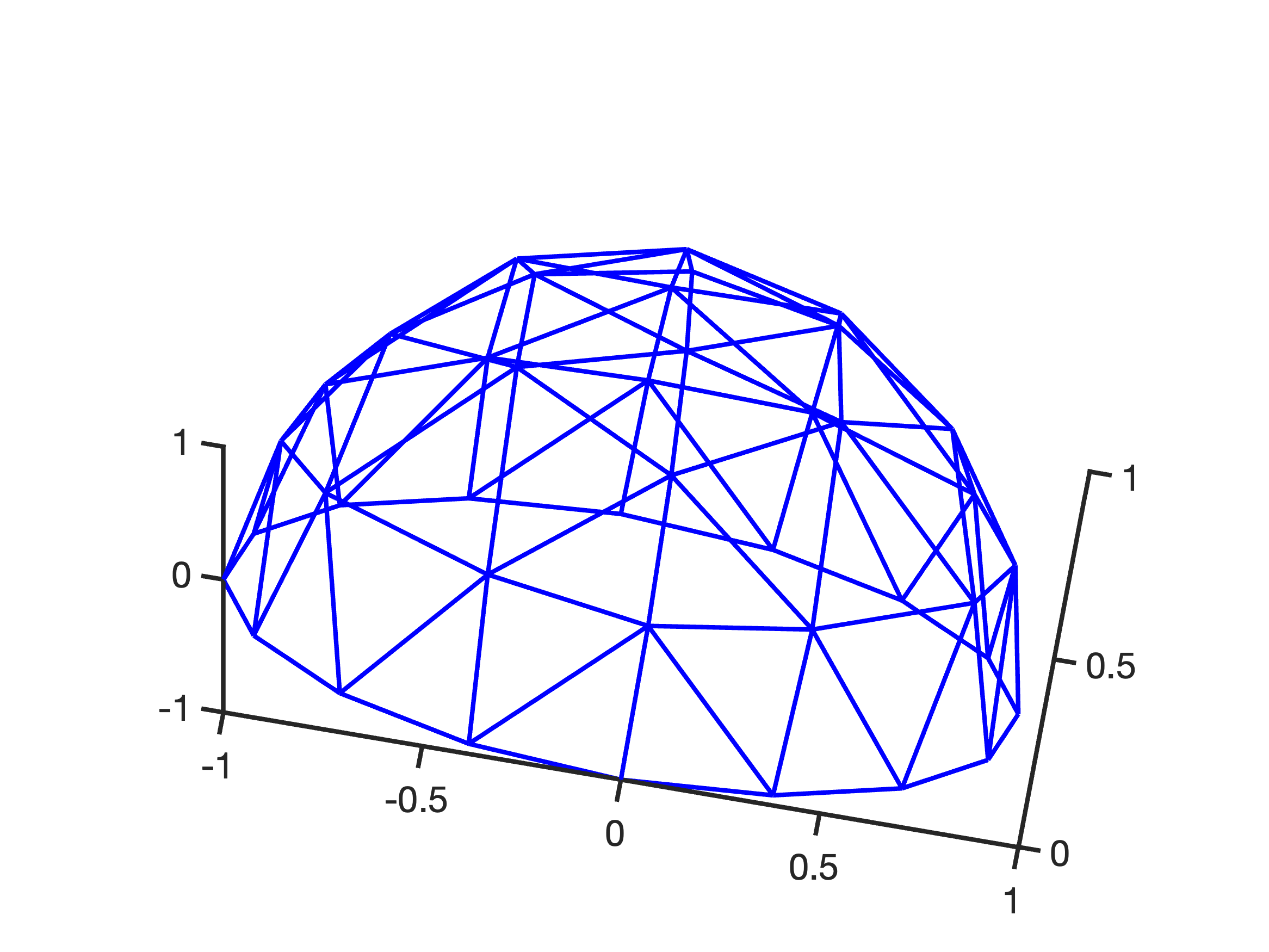}
	\caption{Left: geometry of the lattice problem in the checkerboard domain. White and gray areas are characterized by the optical parameters $\sigma_s=10$ and $\sigma_a=0.01$, while in the black zones $\sigma_s=0$ and $\sigma_a=1$. The internal source of radiation is $q=1$ in the gray square, $q=0$ outside of it. Right: Sketch of the spherical grid for the upper half sphere.
    \label{fig:checkerboard}}
\end{figure}

\subsection{\texorpdfstring{Minimization over $W_{h,N}^c$}{Minimization over W1}}
We investigate the performance of the residual minimization strategy when using the subspace $W_{h,N}^c$ defined in \cref{eq:W1} for different anisotropy parameters $g$. Furthermore, we investigate the behavior on the parameter $K$ in \cref{def:YK}. Here, we choose $K=1,6,15$ such that it corresponds to the number of even spherical harmonics of degree at most $l=0,2,4$, respectively. This choice is motivated by the observation that the $2l+1$ spherical harmonics of degree $l$ are eigenfunctions of $\Theta$ with eigenvalue $g^l$. For comparison, we also run the plain source iteration, cf. \cref{eq:SI}, which corresponds to the choice $K=0$ and $W_{h,N}=\{0\}$, respectively.

The resulting numbers of iterations are displayed in \Cref{tab:W1}.
We observe that the required number of iterations are considerably smaller if residual minimization is performed, i.e., comparing the cases $K>0$ and $K=0$.
As can also be seen from the decay of the residuals depicted in \Cref{fig:W1_residuals}, the minimization-based approach reduces the residual much faster than $\rho^k$ with theoretical rate $\rho=0.999$.
We note that also the plain source iteration, $K=0$, reduces the residual monotonically and slightly faster than predicted by theory, which might be explained by discretization effects, see also \Cref{tab:W1_grid_ref} (left), and the finite diameter of the domain $\R$ \cite{ES13}.
Increasing $K$ yields smaller iteration counts. 
In view of the computational complexity considerations made in \Cref{rem:cost_minimization}, the subspace correction computed here requires fewer floating point operations if the contraction rate of the residuals is better than $\rho^2= 0.998$, because $\mathrm{dim} (W_{h,N}^c) = 1$ for $K>0$. 
This is the case for all our experiments, as it is shown by the numbers in brackets in \cref{tab:W1}, indicating the maximum observed contraction rate $\max_k \normM{\R_h(u_k)}/\normM{\R_h(u_{k-1})}$ during the corresponding iterations. 

\begin{table}[ht]
    \centering \footnotesize
    \setlength{\tabcolsep}{8pt} 
    \renewcommand{\arraystretch}{1.2} 
    \caption{Iteration counts for minimization over $W_{h,N}^c$, defined in \cref{eq:W1}, for different anisotropy parameters $g$ and order $K$ of subspace correction, cf. \cref{def:YK}. Observe that $K=0$ corresponds to the source iteration without subspace correction, cf. \cref{eq:SI}. In brackets, $\max_k \normM{\R_h(u_k)}/\normM{\R_h(u_{k-1})}$. \label{tab:W1}} 
    \begin{tabular}{r r r r r r r r r r}
        \toprule
        & \multicolumn{6}{c}{$g$} \\
        \cmidrule{2-7}	
        $K$ & 0.1 & 0.3 & 0.5 & 0.7 & 0.9 & 0.99 \\
	    \midrule
        0 & 1018 (0.990) & 863 (0.988) & 701 (0.984) & 531 (0.979) & 358 (0.967) & 821 (0.989) \\
        1  & 62 (0.817)   & 58 (0.805)  & 54 (0.795)  &	57 (0.809)	& 97 (0.882)  & 455 (0.981)	\\
	    6  & 20 (0.532)	  &	25 (0.603)	& 14 (0.377)  &	13 (0.371)	& 31 (0.706)  & 283 (0.973)	\\
	    15 & 12 (0.312)	  &	10 (0.245)	& 8 (0.157)	  &	8 (0.216)	& 18 (0.657)  & 234	(0.969)\\
	    \bottomrule
    \end{tabular}
\end{table}

\begin{table}[ht]
    \centering \footnotesize
    \setlength{\tabcolsep}{8pt} 
    \renewcommand{\arraystretch}{1.2} 
    \caption{Average times (in seconds) per iteration for minimization over $W_{h,N}^c$, defined in \cref{eq:W1}, for different anisotropy parameters $g$ and order $K$ of subspace correction, cf. \cref{def:YK}, with $K=0$ the source iteration without subspace correction. In brackets, total time (in hours) taken by the iterative procedure. \label{tab:W1-timings}}
      \begin{tabular}{r r r r r r r r r r}
        \toprule
        & \multicolumn{6}{c}{$g$} \\
        \cmidrule{2-7}	
        $K$ & 0.1                           & 0.3                           & 0.5                           & 0.7 & 0.9 & 0.99 \\
	    \midrule
        0   &  \SI{66}{s} (\SI{18.72}{h})   &  \SI{66}{s} (\SI{15.93}{h})   &  \SI{68}{s} (\SI{13.32}{h})   &  \SI{67}{s} (\SI{9.98}{h}) & \SI{66}{s} (\SI{6.63}{h}) & \SI{69}{s} (\SI{15.87}{h}) \\
        1   & \SI{156}{s}  (\SI{2.70}{h})   & \SI{153}{s}  (\SI{2.48}{h})   & \SI{153}{s} (\SI{2.30}{h})   &	\SI{155}{s} (\SI{2.46}{h})	& \SI{153}{s} (\SI{4.14}{h})  & \SI{155}{s} (\SI{19.59}{h})	\\ 
	    6   & \SI{151}{s}  (\SI{0.84}{h})	& \SI{158}{s}  (\SI{1.10}{h})	& \SI{157}{s} (\SI{0.61}{h})   &	\SI{159}{s} (\SI{0.57}{h})	& \SI{156}{s} (\SI{1.35}{h})  & \SI{160}{s} (\SI{12.65}{h})	\\
	    15  & \SI{163}{s}  (\SI{0.54}{h})	& \SI{168}{s}  (\SI{0.47}{h})	& \SI{170}{s} (\SI{0.38}{h})   &	\SI{168}{s} (\SI{0.37}{h})	& \SI{164}{s} (\SI{0.82}{h})  & \SI{167}{s} (\SI{10.86}{h})\\
	    \bottomrule
    \end{tabular}
\end{table}
As discussed in \Cref{rem:cost_minimization} each iteration of the residual minimization scheme requires more operations than the source iteration. Therefore, we also compare the runtime of these methods using our implementation. In order to solve the linear systems in \cref{eq:even_parity} and \cref{eq:corr_half_step_error_equation_discrete} we use sparse LU-factorizations, which are precomputed before entering the iteration. For finer discretizations than the ones we use here, such a pre-computation might be restrictive in terms of memory. To solve \cref{eq:corr_half_step_error_equation_discrete} more efficiently for larger systems, we refer to \cite{DPS22} for a conjugate gradient method with a multigrid preconditioner. 
In any case, improved solvers for \cref{eq:even_parity} are equally beneficial for both the usual source iteration as well as the residual minimization based approach.
In \Cref{tab:W1-timings} we display the average runtime per step as well as the overall runtime for the considered methods.
These timings match the theoretical considerations of \Cref{rem:cost_minimization}.
Moreover, the computation time required for computing the solution to \cref{eq:corr_half_step_error_equation_discrete}, has a runtime less than one solve of \cref{eq:even_parity}, which can be seen from comparing the rows for $K>0$ with the row $K=0$.
Because of these observations, we do not show further timings but use computational complexity considerations to discuss the different methods in the following numerical experiments.

\begin{table}[ht]
    \centering \footnotesize
    \setlength{\tabcolsep}{8pt} 
    \renewcommand{\arraystretch}{1.2} 
    \caption{Iteration counts for minimization over $W_{h,N}^c$, defined in \cref{eq:W1}, for $g=0.7$ and $K=0$ (left) and $K=6$ (right), and different spatial and angular grids.
    \label{tab:W1_grid_ref}}
\begin{tabular}{r c c c c c c c}
        \toprule
        $K=0$& \multicolumn{4}{c}{$\nse$} \\
    	\cmidrule{2-5}	
        $\nre$  & 16 & 64 & 256 & 1024 \\
	    \midrule
        841     &  475 & 483  & 487  & 488\\
	    3\,249  &  499 & 509  & 514  & 515\\
	    12\,769 &  509 & 520 &  525 & 526\\
        50\,625 &  513 & 524 & 529  & 531\\
	    \bottomrule
    \end{tabular}
    \qquad\qquad
    \begin{tabular}{r c c c c c c c}
        \toprule
        $K=6$& \multicolumn{4}{c}{$\nse$} \\
    	\cmidrule{2-5}	
        $\nre$  & 16 & 64 & 256 & 1024 \\
	    \midrule
        841     & 8  & 8  & 9   & 9    \\
	    3\,249  & 8  & 9  & 10  & 10   \\
	    12\,769 & 8  & 10 & 11  & 11   \\
        50\,625 & 9  & 11 & 12  & 13   \\
	    \bottomrule
    \end{tabular}
\end{table}
Before testing the performance of the approach on the next subspace for minimization, let us discuss \Cref{tab:W1_grid_ref}, where we show the iteration counts for the choices $g=0.7$ and $K=0$ (left) and $K=6$ (right), and for different angular and spatial grids obtained by successive refinements. We observe that the number of iterations varies only slightly upon mesh refinement, which we expect because we derived the iteration from its infinite-dimensional counterpart.

\begin{figure}[ht]
	\centering \small 
    \setlength\tabcolsep{0.55em}
	%
	\caption{Residual decay for correction searched in the subspace $W_{h,N}^c$. From left to right, we plot $\normM{\R_h(u_k)}$; first row $g=0.1,\ 0.3,\ 0.5$; second row, $g=0.7,\ 0.9,\ 0.99$. The lines' style refers to number of even eigenfunctions of the scattering operator employed: $K=0$ (standard source iteration) densely dotted line; $K=1$ solid line; $K=6$ dotted line; $K=15$ dashed line.
    \label{fig:W1_residuals}}
\end{figure}

\subsection{\texorpdfstring{Minimization over $\tsup[1]{W}_{h,N}^c$}{Minimization over tilde W1}}
As a second test case, we study how the residual minimization approach performs over the enriched subspace $\tsup[1]{W}_{h,N}^c$ defined in \cref{eq:tW1}.
Since the source iteration \cref{eq:SI} does not depend on the chosen subspace, the reader may refer to the $K=0$ row in \Cref{tab:W1} for corresponding iteration numbers.
As can be seen from \Cref{tab:tW1}, for moderate values of the anisotropy factor $g$ the improvement in the number of iterations is negligible with respect to minimization over $W_{h,N}^c$. However, for highly forward-peaked scattering, $g=0.99$, which causes the iteration to be notably slower than the other cases, the improvement is more visible, even for low order corrections ($K=1$). 
\Cref{fig:W1_tilde_residuals} shows the convergence history of the residuals, and we observe a robust convergence behavior.
Since $\mathrm{dim}(\tsup[1]{W}_{h,N}^c)=4$, minimization over this subspace is useful in terms of the theoretical computational complexity considerations if the contraction rate of the residuals stays below $\rho^5 = 0.995$, cf. \Cref{rem:cost_minimization}. Once again, our method achieves this requirement, as shown in brackets in \Cref{tab:tW1}. 
Except for $g\in \{0.9,0.99\}$ and $K=1$, we also have that $(\rho_{SI}^g)^5>\max_k \normM{\R_h(u_k)}/\normM{\R_h(u_{k-1})}$ with $\rho_{SI}^g$ denoting the observed (worst) contraction rate of the source iteration as displayed in \Cref{tab:W1}, i.e., minimization over $\tsup[1]{W}_{h,N}^c$ is more efficient.

\begin{table}[ht]
    \centering \footnotesize
    \setlength{\tabcolsep}{8pt} 
    \renewcommand{\arraystretch}{1.2} 
    \caption{Iteration counts for minimization over $\tsup[1]{W}_{h,N}^c$ defined in \cref{eq:tW1} for different anisotropy parameters $g$ and order $K$ of subspace correction, cf. \cref{def:YK}.  In brackets $\max_k \normM{\R_h(u_k)}/\normM{\R_h(u_{k-1})}$. \label{tab:tW1}}

	\caption{Residual decay for correction searched in the subspace $\tsup[1]{W}_{h,N}^c$. From left to right, we plot $\normM{\R_h(u_k)}$; first row $g=0.1,\ 0.3,\ 0.5$; second row, $g=0.7,\ 0.9,\ 0.99$. The lines style refers to the dimension of the corrections subspace: $K=1$ solid line; $K=6$ dotted line; $K=15$ dashed line.\label{fig:W1_tilde_residuals}}
\end{figure}

\subsection{\texorpdfstring{Minimization over $\tsup[1]{W}_{h,N}^{c,m}$}{Minimization over WN}}
Exploiting Anderson-type acceleration techniques as described in \Cref{sec:Enriched_space+AA} we observe a substantial reduction in the iteration count for all values of $g$ and already for moderate $m$ and low-order corrections. Indeed, comparing \Cref{tab:tW1} and \Cref{tab:tWN2}, where a history of $m=2$ iterates is taken into account for residual minimization and thus $\mathrm{dim}(\tsup[1]{W}_{h,N}^{c,2}) = 6$, we notice that already for $K=1$ the number of iterations is roughly reduced by a factor of $3$ for small $g$ and a factor of $2$ for $g$ close to $1$. The numbers in \Cref{tab:tWN4}, where $m=4$ was chosen, are comparable to those in \Cref{tab:tWN2}. 
Thus, we prefer the choice $m=2$ because it requires less memory. 
The computational complexity does not increase with $m$ in practice because the residuals of the previous iterates can be stored in memory.
Therefore, although $N=6$ and $N=8$ for $m=2$ and $m=4$, respectively, the number of residual computations is as in the previous subsection, i.e., $5$ per step.
Since for $g<0.99$ or $K>1$ both $\rho^5$ as well as $(\rho_{SI}^g)^5$ are larger than the observed contraction rates as shown in brackets in \Cref{tab:tWN2} and \Cref{tab:tWN4}, the proposed methodology is more efficient than the source iteration.

\begin{table}[ht]
    \centering \footnotesize
    \setlength{\tabcolsep}{8pt} 
    \renewcommand{\arraystretch}{1.2} 
    \caption{Iteration counts for minimization over $\tsup[1]{W}_{h,N}^{c,2}$ defined in \cref{eq:WN_def} for different anisotropy parameters $g$ and order $K$ of subspace correction, cf. \cref{def:YK}.  In brackets $\max_k \normM{\R_h(u_k)}/\normM{\R_h(u_{k-1})}$.
    \label{tab:tWN2}}
\begin{tabular}{r c c c c c c c c c}
        \toprule
        $m=2$ & \multicolumn{6}{c}{$g$} \\
        \cmidrule{2-7}	
        $K$ & 0.1        & 0.3        & 0.5        & 0.7        & 0.9        & 0.99 \\
	   \midrule
        1   & 17 (0.500) & 17 (0.483) & 20 (0.579) & 27 (0.667) & 42 (0.778) & 262 (0.975) \\
	    6   & 9 (0.224)  & 10 (0.237) & 8 (0.172)  & 9 (0.210)  & 19 (0.518) & 96 (0.937)\\
	    15  & 7 (0.119)  & 7 (0.100)  & 6 (0.076)  & 7 (0.121)  & 13 (0.375) & 67 (0.838) \\
	   \bottomrule
    \end{tabular}
\end{table}

\begin{table}[ht]
    \centering \footnotesize
    \setlength{\tabcolsep}{8pt} 
    \renewcommand{\arraystretch}{1.2} 
    \caption{Iteration counts for minimization over $\tsup[1]{W}_{h,N}^{c,4}$ defined in \cref{eq:WN_def} for different anisotropy parameters $g$ and order $K$ of subspace correction, cf. \cref{def:YK}.  In brackets $\max_k \normM{\R_h(u_k)}/\normM{\R_h(u_{k-1})}$. \label{tab:tWN4}}
    %
	\caption{Residual decay for correction searched in the subspace $\tsup[1]{W}_{h,N}^{c,2}$. From left to right, we plot $\normM{\R_h(u_k)}$; first row $g=0.1,\ 0.3,\ 0.5$; second row, $g=0.7,\ 0.9,\ 0.99$. The lines' style refers to number of even eigenfunctions of the scattering operator employed: $K=1$ solid line; $K=6$ dotted line; $K=15$ dashed line.
    \label{fig:Wc2_residuals}}
\end{figure}

\begin{figure}[ht]
	\centering \small 
    \setlength\tabcolsep{0.55em}
	%
	\caption{Residual decay for correction searched in the subspace $\tsup[1]{W}_{h,N}^{c,4}$. From left to right, we plot $\normM{\R_h(u_k)}$; first row $g=0.1,\ 0.3,\ 0.5$; second row, $g=0.7,\ 0.9,\ 0.99$. The lines' style refers to number of even eigenfunctions of the scattering operator employed: $K=1$ solid line; $K=6$ dotted line; $K=15$ dashed line.
    \label{fig:Wc4_residuals}}
\end{figure}

\section{Conclusions and discussion}
\label{sec:Conclusion}
In this paper, we have developed a generic and flexible strategy to accelerate the source iteration for the solution of anisotropic radiative transfer problems using residual minimization. 
We showed convergence of the resulting method for any choice of subspace employed in the residual minimization.
The flexibility in choosing the subspace was used to exploit high-order diffusion corrections, which were shown to be effective for highly forward-peaked scattering.
Moreover, the numerical results confirmed that the required iteration counts do depend on the discretization only mildly.

We mention that our approach can be seen as a two-level scheme. We leave it to future work to extend it and compare it to angular multilevel schemes. Moreover, the analysis of the precise approximation properties of the considered subspaces is also left to future research. 
We close by mentioning that the efficient and robust solution of the source problem \cref{eq:RTE} is also relevant for solving eigenvalue problems, see, e.g., \cite{CFWBC13, BA22, SK24}.

\section*{Acknowledgements}
R.B. and M.S. acknowledge support by the Dutch Research Council (NWO) via grant OCENW.KLEIN.183.

\bibliographystyle{siamplain}
\bibliography{references_FINAL}

\end{document}